\documentclass[12pt]{article}
\usepackage{amssymb}
\usepackage{graphicx}

\newtheorem{lemma}{Lemma}
\newtheorem{assumption}{Assumption}
\newtheorem{theorem}{Theorem}
\newtheorem{remark}{Remark}
\newenvironment{proof}{{\em Proof:\/}}{}

\newcommand\1{{\bf 1}}

\begin{document}

\title{First hyperbolic times for intermittent maps with unbounded derivative\thanks{The work of CB is supported by an NSERC grant}
}
\author{Christopher Bose\thanks{Department of Mathematics and Statistics, University of Victoria, PO Box 3060 STN CSC,
Victoria BC Canada V8W3R4} \and Rua Murray\thanks{Department of Mathematics and Statistics, University of
Canterbury, Private Bag 4800, Christchurch 8140, New Zealand.}}

\maketitle

\begin{abstract}
We establish some statistical properties of the {\em hyperbolic times\/} for
a class of nonuniformly expanding dynamical systems. The maps arise as factors of area
preserving maps of the unit square via a geometric Baker's map type construction,
exhibit intermittent dynamics, and have unbounded derivatives. 
The geometric approach captures various examples from the literature over the 
last thirty years.
The statistics of these maps are controlled by the order of tangency that a certain ``cut
function'' makes with the boundary of the square. Using
a large deviations result of Melbourne and Nicol we obtain sharp estimates
on the distribution of first hyperbolic times. 

As shown by Alves, Viana and others, knowledge of the tail of the distribution of first
hyperbolic times leads to estimates on the rate of decay of correlations and derivation of a CLT.    For our family  of maps, we compare the estimates on correlation decay rate and CLT derived via hyperbolic times with those
derived by a direct Young tower construction. The latter estimates are known to be sharp. 

\end{abstract}

Let $f:X\to X$ be a dynamical system which is expanding on average, but not
necessarily uniformly with every time-step. Amongst the important questions to ask
about $f$ are: is there an invariant SRB-probability measure? how quickly do
correlations between observables decay under iteration by~$f$? does a central
limit theorem hold? are these properties stable to perturbations of $f$?
When $f$ is {\em uniformly expanding\/} the answers to these
questions are well understood (see eg~\cite{BaY,HK}), but the situation for
non-uniformly expanding $f$ is more delicate. Difficulties arise from
the fact that orbits of $f$ may experience periods of local contraction as well
as expansion (for example, quadratic maps~\cite{BY1}), rapidly varying
derivatives near singularities leading to unbounded distortion (eg~\cite{ALP}), or indifferent fixed points~\cite{Y2}.

The theory of \emph{hyperbolic times} has proved useful for analysing the
statistical properties of non-uniformly expanding maps~\cite{A,AA,ABV,ALP,V}.
 The idea was introduced
in~\cite{A2} to  handle specific non-uniformly expanding families
(Alves-Viana maps~\cite{V} and certain quadratic maps~\cite{F}), and has since been
developed for various non-uniformly expanding and partially hyperbolic
classes~\cite{ABV}. Gou\"ezel~\cite{G2} has used hyperbolic times to show that the Alves-Viana maps
exhibit stretched exponential decay of correlations. Alves, Luzzatto and Pinheiro~\cite{ALP}
prove exponential decay of correlations (and a CLT) for a class of
non-uniformly expanding maps by
using hyperbolic times and a Young~tower
construction. Further results are obtained in~\cite{ALP2} for one-dimensional families.
A survey paper discussing many of these ideas is found in~\cite{A}.

Roughly speaking, hyperbolic times are defined as follows\footnote{The exact definition is detailed in equation (\ref{eqn:defHT}) in Section \ref{sec.2}.}: Given an 
orbit $\{f^k x\}$ of a point $x \in X$,  an integer $n >0$ is a \emph{hyperbolic time} for $x$ if for all 
$1 \leq l  \leq n$ the cumulative derivative $\Pi_{k=n-l}^{k=n-1}| Df (f^l (x))| $ grows exponentially in $l$.  In addition, if the map has a nonempty set  $\mathcal S$ of singular points we require 
the distance from $f^{n-l}(x)$ to $\mathcal S$ to contract at exponential rate in 
$l$, essentially an exponential escape condition.  These exponential rates are to be chosen uniformly for $x\in X$. Note that for uniformly expanding maps with bounded distortion both conditions automatically hold and every $n$ is a hyperbolic time.

Therefore, the idea is to choose certain times at which the accumulated expansion and escape from the singular set mimic the uniformly expanding case even though there may have been times along the way where these properties failed.  In this way, many good statistical properties can be recovered.

There are at least two important statistics associated
with hyperbolic times: their long-run frequency of
occurrence, and the distribution of \emph{first hyperbolic times}. Obtaining precise
quantitative control of the distribution of hyperbolic times can be an
important step~\cite{ALP,ALP2} in further analysis of 
statistical properties of the map, including the above-mentioned rates of decay of correlation and
CLT. 

In~\cite{AA} a map~$F$ on the interval $[-1,1]$ is introduced that has positive density of hyperbolic
times, but for which the first hyperbolic time fails to be integrable. $F$ has a number
of special properties (symmetry, preservation of Lebesgue measure, and a pair of indifferent
fixed points with quadratic tangencies). In this paper we present a class of interval
maps~$\mathcal{C}_\alpha$ (parametrised by\footnote{And certain 
continuous functions on $[0,1]$.} $\alpha\in(0,\infty)$) 
which arise as nonuniformly expanding one-dimensional factors of
geometrically derived {\em generalized Baker's transformations\/} (GBTs)~\cite{B1}. Each
map in $\mathcal{C}_\alpha$ has an indifferent fixed point 
(IFP) at $0$, and in fact the map $F$ in \cite{AA} is conjugate to a certain map $f_1\in\mathcal{C}_1$.
Each $f\in\mathcal{C}_{\alpha}$ has a positive long-run frequency of hyperbolic times (by an 
argument from~\cite{ABV}), and integrability of first hyperbolic times holds when $\alpha \in (0,1)$.  However, as $\alpha$
increases through~$1$ this integrability is lost (Lemma~\ref{prop:HtimeLower}). 
In this way, $\alpha=1$
appears as a transition  point for our families $\mathcal{C}_{\alpha}$.

 As becomes clear in our proof of Lemma \ref{prop:HtimeLower}, the 
non-integrability is entirely due to lower bounds on the first hyperbolic time which are determined 
by escape statistics from the neighbourhood  of the 
IFP(s). These same escape statistics are then used to provide upper bounds
on the first hyperbolic times in Theorem \ref{th:HintCLT}, completing the analysis and providing sharp estimate on tail asymptotics for hyperbolic times for the range $0<\alpha < 1$. 
While precise statements are given below, roughly speaking: if
$m$ denotes Lebesgue measure and $h(x)$ denotes the first hyperbolic time on an orbit beginning at~$x$
then $m\{x~:~h(x)\geq n\} \sim n^{-1/\alpha}$.

In \cite{ALP2, ALP} hyperbolic times asymptotics are used to estimate correlation decay rates and to establish CLT's for nonuniformly hyperbolic systems. When applied to our family $\mathcal{C}_{\alpha}$, these results imply upper bounds on correlation decay rates
that fail to be sharp,  compared to the direct computation via Young towers detailed in \cite{BM2010b}.  The range of parameters in our family leading to a CLT is similarly underestimated by the hyperbolic times analysis. Remark \ref{rem:compare} at the end of Section~\ref{sec.1} provides a comparative analysis of these two approaches.


The class $\mathcal{C}_\alpha$ is presented in Section~\ref{sec.1}, 
hyperbolic times are reviewed and lower bounds are derived in Section~\ref{sec.2} and the upper bounds 
are established (via a large deviations principle of Melbourne and Nicol~\cite{MN} on a suitable Young tower~\cite{Y2}) 
in Section~\ref{sec.3}. In Section \ref{sec.concl} we discuss our results in the context of the existing literature on hyperbolic times.  Some technical estimates are placed in 
 an appendix (Section~\ref{sec.4}).

{\em Notation:\/} We write $f(n)=O(g(n))$ to mean there is a constant $C<\infty$ such that
$f(n)\leq C\,g(n)$ and $f(n) \sim g(n)$ to mean $f(n)=O(g(n))$ and $g(n)=O(f(n))$.

\section{Generalized baker's transformations and $\mathcal{C}_{\alpha}$}\label{sec.1}

The generalized baker's construction~\cite{B1} defines a large class of invertible, 
Lebesgue-measure-preserving maps of the unit square $S=[0,1]\times[0,1]$. Specifically,
a two-dimensional map~$B$ on $S$ is determined by a measurable \emph{cut function} $\phi$ on $[0,1]$ satisfying $0 \leq \phi \leq 1$.  The graph
$y=\phi(x)$ partitions the square $S$ into upper and lower pieces and the line 
$\{x=a\}$, where $a=\int_0^1\phi(t)\,dt$, partitions the square into a `left half' $[0,a] \times [0,1]$ and 
a `right half' $[a,1] \times [0,1]$.  The \emph{generalized baker's transformation} (GBT) $B$ maps the left half into the lower piece and the right half into the upper piece in such a way that:

\begin{itemize}
\item Vertical lines in the left (right) half are mapped affinely into vertical `half lines' under (over) the graph of the cut function $\phi$.
\item $B$ preserves two-dimensional Lebesgue measure. 
\item The factor action $f$ of $B$ restricted to vertical lines is (conjugate to) a piecewise monotone increasing, Lebesgue-measure-preserving interval map on $[0,1]$ with two monotonicity pieces $[0,a]$ and 
$[a,1]$. 
\end{itemize}

The action of a typical GBT is presented in Figure ~\ref{fig:gbt}.
 
\begin{figure}
\begin{center}
\includegraphics[width= 0.80\textwidth ]{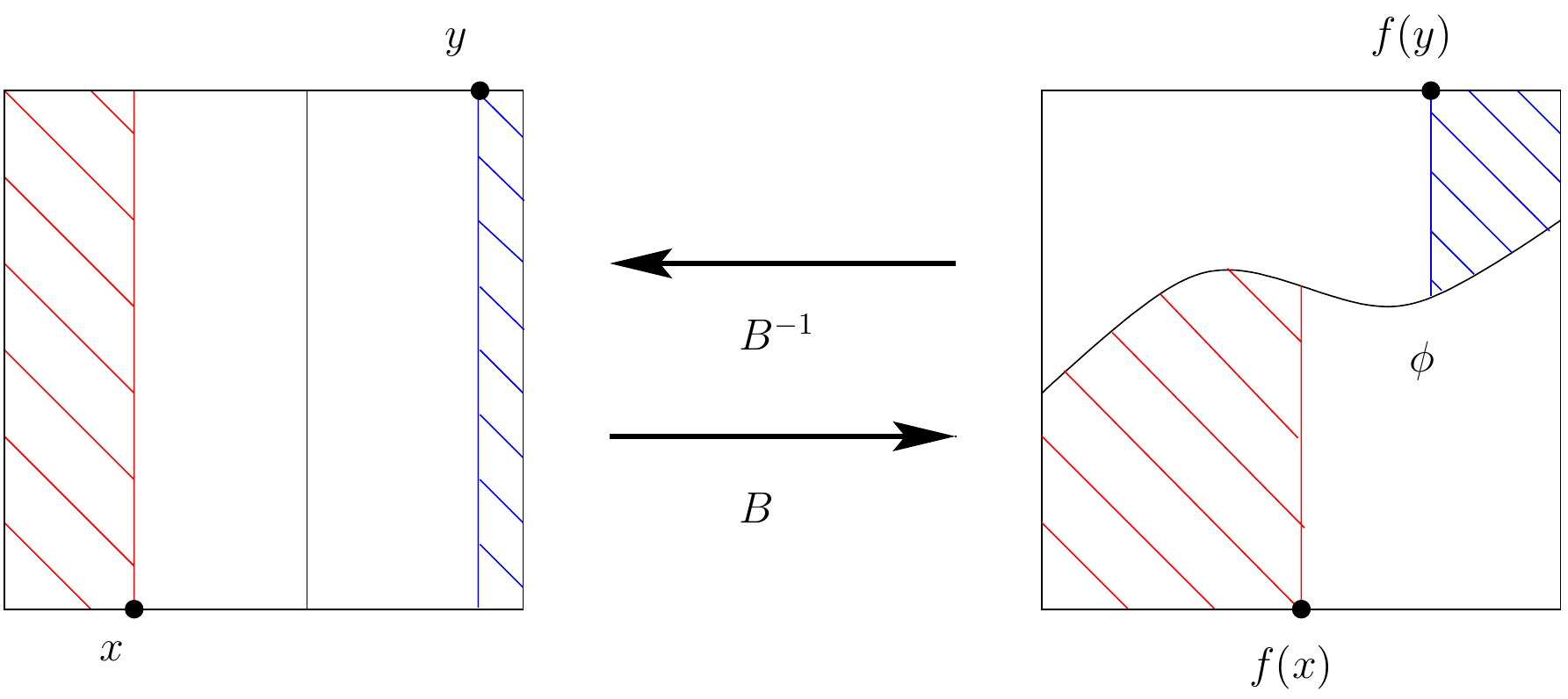}
\caption{The GBT}
\label{fig:gbt}
\end{center}
\end{figure}

When $\phi \equiv 1/2$ the map is the classical baker's transformation where the action 
on vertical lines is an affine contraction and the 
map $f$ is $x \rightarrow 2x\pmod{1}$. On the other hand, 
every measure-preserving transformation $T$ on a
(nonatomic, standard, Borel) probability space with entropy satisfying $0< h(T) < \log 2$ is measurably isomorphic to
some generalized baker's transformation on the square $S$ (see~\cite{B1}). 

In order to proceed, we establish some notation. 
Each GBT~$B$ has a skew-product form $B(x,y)=(f(x),g(x,y))$ where
\begin{equation}\label{eqn.map1}
g(x,y)=\left\{\begin{array}{ll}\phi(f(x))\,y & x\leq a,\\ y+\phi(f(x))\,(1-y)&x>a,\end{array}\right.
\qquad\mbox{and}\qquad \left\{\begin{array}{ll}x=\int_0^{f(x)}\phi(t)\,dt & x\leq a,\\ 1-x=\int_{f(x)}^1(1-\phi(t))\,dt &x>a,\end{array}\right.
\end{equation}
defines $f$ implicitly. 

Furthermore, by construction the vertical lines $\{ x=0\},\, \{x=1\}$ are mapped into themselves by $B$ so that $0,1$ are fixed points of $f$.
If $\phi$ is continuous then $f$ is differentiable on both $(0,a)$ and $(a,1)$ and
\begin{equation}
\label{eqn.mapd}
\frac{df}{dx} = \left\{\begin{array}{ll} \frac{1}{\phi(f(x))}&x<a\\\frac{1}{1-\phi(f(x))}&x>a.\end{array}\right.
\end{equation}
Since $\phi(t)\in[0,1]$ for each $t\in[0,1]$, $\frac{df}{dx}\geq 1$, so $f$ is expanding, each branch of $f$ is increasing, and
may have infinite derivative at preimages of places where $\phi(t)\in\{0,1\}$. We will call $f$ the {\bf expanding factor\/} of $B$.
Note also that if $\phi$ is a {\em decreasing\/} function then $\frac{df}{dx}$ is increasing on $(0,a)$ and decreasing on $(a,1)$.

\begin{lemma}[Properties of GBTs]\label{lem.setup}
Let $f$ be defined by~(\ref{eqn.map1}). Each $x\in [0,1]$  has two 
preimages under $f$: $x_l<a$ and $x_r>a$ and moreover
\begin{enumerate}
\item[(i)] For every $x \in [0,1]$,
\begin{equation}
\label{eqn.map2}
 x_l = \int_0^x\phi(t)\,dt \textnormal{ and } x_r-a =\int_0^x 1 - \phi(t) \, dt = x - x_l;
 \end{equation}
\item[(ii)] $\frac{dx_l}{dx}= \phi$ and  $\frac{dx_r}{dx} = 1- \phi$ Lebesgue almost everywhere;
\item[(iii)]  Lebesgue measure~$m$ is $f$--invariant.
\end{enumerate}
\end{lemma}

\begin{proof} See appendix. \qquad$\Box$
\end{proof}

\subsection*{A class $\mathcal{C}_{\alpha}$ of expanding factors of GBTs}

Let $\alpha\in(0,\infty)$. Maps $f\in\mathcal{C}_{\alpha}$ arise as expanding factors of
GBTs whose cut functions $\phi$ satisfy
\begin{itemize}
\item $\phi$ is continuous and decreasing function on $[0,1]$ with $0 \leq \phi \leq 1$. 
\item there is a constant $c_0$ and a $C^1$ function $g_0$ on $(0,1)$ such that near $t=0$
$$\phi(t) = 1-c_0\,t^\alpha + g_0(t)$$
where $\frac{dg_0}{dt}=o(t^{\alpha-1})$;
\item either $\phi(1)>0$ or there are constants $c_1\in(0,\infty)$, $\alpha'\leq \alpha$ and
 a $C^1$ function $g_1$ on $(0,1)$ such that near $t=1$
$$\phi(t) = c_1\,(1-t)^{\alpha'} + g_1(1-t)$$
where $\frac{dg_1}{dt}=o(t^{\alpha'-1})$.
\end{itemize}

It follows from these conditions that $\phi$ is $C^1$ on $(0,1)$
and therefore each $f\in \mathcal{C}_\alpha$ has two piecewise increasing $C^2$ branches
$f_l, f_r$ with respect to the partition into intervals $(0, a)$
and $(a,1)$ (where $a= a(\phi)$).  The branch $f_l$ has continuous extension to $[0,a]$ 
(similarly for $f_r$ and $[a,1]$) and $f_r^\prime(x) \rightarrow \infty$ as $x \rightarrow a^+$.  

Near $x=0$ each $f\in\mathcal{C}_\alpha$ has the formula
$$f(x) = x + \frac{c_0}{1+\alpha}\,x^{1+\alpha} + o(x^{1+\alpha}),$$
giving an {\em indifferent fixed point\/} (IFP) at $0$.
If $\phi(1)=0$ then $f$ also has an IFP at $x=1$, and the order of
tangency of the graph of $f$ near~$1$ is $O((1-t)^{1+\alpha^\prime})$. For maps with IFPs 
at both $0$ and $1$ where the order of tangency is higher at $1$ then $0$, the conjugacy
$\phi(t)\mapsto 1 - \phi(1-t)$ will put the higher order tangency at $0$. There is thus no loss of
generality in assuming that the ``most indifferent'' point is at $x=0$. In case $\phi(1)>0$, 
equation~(\ref{eqn.mapd}) shows that the fixed point at $1$ is hyperbolic.

{\bf Example 1  [Alves-Ara\'ujo map $F$].\/} See \textnormal{\cite{AA}}. Let $\phi(t)=1-t$. Then $\alpha=\alpha'=1=c_0=c_1$ and $g_0=g_1=1$. Then
$$f_1(x):=\left\{\begin{array}{ll}1-\sqrt{1-2\,x}&x<1/2,\\\sqrt{2x-1}&x>1/2.\end{array}\right.$$
In~\cite{R}, Rahe established that the map $B$ is Bernoulli (using techniques from~\cite{FO,dJ-R}).
Moreover, $f_1\in\mathcal{C}_{1}$ is conjugate (by an affine scaling $[0,1]\to[-1,1]$) to a map presented in~\cite{A,AA}
which has non-integrable first hyperbolic time. Despite this, $f_1$ exhibits polynomial decay of
correlations for H\"older observables with rate $O(1/n)$~\cite{BM2010b}.\qquad$\Box$

{\bf Example 2 [Symmetric case].\/} Let $\phi(t)=1-(2t)^\alpha/2$ for $t\in[0,1/2]$ and $\phi(t)=1-\phi(1-t)$ for $t\in[1/2,1]$.
Then $\phi$ is symmetric (and $\alpha=\alpha'$); let the expanding factor of the corresponding GBT be
denoted by $f_\alpha$. Then $f_\alpha$ has indifferent fixed points at $0$ and $1$ with tangency of
order $(1+\alpha)$; moreover $\frac{df}{dx}\to\infty$ as $x\to\frac{1}{2}$.
In \cite{BM2010b} it is shown that H\"older continuous functions
have correlation decay rate $O(n^{-1/\alpha})$ under $f_\alpha$. The paper \cite{C-H-V} obtains similar results for a conjugate
class of maps on $[-1,1]$. The results in this paper, Theorem~\ref{prop:HtimeLower} and  
Theorem~\ref{th:HintCLT}, imply that the {\em first hyperbolic time\/} is integrable if and only if $\alpha<1$,
showing that the map $f_1$ emerges as an interesting transition point in
the class $\mathcal{C}_{\alpha}$.

\begin{remark} 
The class of examples discussed here actually have a rather long history in the mathematical physics literature. For other examples see
\textnormal{\cite{Hem84,GrHo85,Pik91,Zwe98,ArtCri04}}.
\end{remark}

\subsection*{A useful dynamical partition}

To analyse $f\in\mathcal{C}_{\alpha}$ we make a convenient partition of $[0,1]$.
First, observe that since $f^2$ has four $1$--$1$ and onto branches, $f$ admits a period-$2$ orbit
$\{x_0,y_0\}$ where $0<x_0<a<y_0<1$. Next, for each $n>0$ let $x_n=f^{-1}(x_{n-1})\cap(0,a)$
and $x_n^\prime = f^{-1}(x_{n-1})\cap(a,1)$. Then
$$ 0  \cdots < x_{n+1}<x_n<\cdots<x_0 < a\qquad\mbox{and}\qquad
a < \cdots < x_n^\prime < \cdots < x_2^\prime < y_0 < 1.$$
Defining  $y_n=f^{-1}(y_{n-1})\cap(a,1)$ and
$y_n^\prime=f^{-1}(y_{n-1})\cap(0,a)$ allows a similar partitioning of $(x_0,a)$ and $(y_0,1)$
(note that $x_1^\prime=y_0$ and $y_1^\prime=x_0$). Put
$$J_n=(x_{n+1},x_{n}), I_n=(x'_{n+1},x'_{n}), J^\prime_n=(y_{n},y_{n+1}), I^\prime_n=(y^\prime_n,y^\prime_{n+1}).$$
These intervals partition $[0,1]$ from left to right as $$0\cdots J_n\cdots J_0 \, x_0\,  I^\prime_1\cdots I^\prime_n\cdots a\cdots I^\prime_1 \, y_0\, J^\prime_0\cdots J^\prime_n\cdots 1$$
with
$$I_k\stackrel{f}{\to} J_{k-1}\stackrel{f}{\to} \cdots J_0 \stackrel{f}{\to} \left(\cup_{l}I_l\cup I^\prime_l\right)$$
(and similarly for the $J^\prime_\cdot, I^\prime_\cdot$ intervals).

\begin{lemma}
\label{l.meas}
For $f\in\mathcal{C}_{\alpha}$, and the notation established above.
\begin{enumerate}
\item[(i)]  $x_n\sim \left(\frac{1}{n}\right)^{1/\alpha}$;
\item[(ii)] $m(J_k)\sim (\frac{1}{k})^{1+1/\alpha}$;
\item[(iii)] for $x \in I_k$, $\frac{df}{dx}\sim k$;
\item[(iv)] $m(I_k)\sim (\frac{1}{k})^{2+1/\alpha}$;
\item[(v)] for $x \in I_k, ~~\textnormal{dist}(x,a) \sim
\left( \frac{1}{k}\right)^{1 + \frac{1}{\alpha}}$.
\end{enumerate}
When $\phi(1)=0$, similar estimates hold for the $\cdot'$ intervals with $\alpha$ replaced by $\alpha^\prime$.
\end{lemma}

\begin{proof} For parts~(i)--(iv) see~\cite[Lemma~1]{BM2010b}; for part~(v) see appendix.
\qquad$\Box$
\end{proof}

\begin{remark}
When $f\in\mathcal{C}_\alpha$ and $\phi(1)>0$ the fixed point at~$1$ is hyperbolic. The corresponding decay rates 
for $m(I^\prime_k), m(J^\prime_k)$ are exponential, and (iii) does not hold. Some of the estimates and 
statements below can be modified in this latter case\footnote{In particular, $\{1\}$ is no longer an exceptional point; see the definition in Section \ref{sec.2}.}.
\end{remark}

\begin{assumption}\label{assum:1}
For the remainder of the paper we assume that $\phi(1)=0$ so that $f$ also has an IFP
at $1$ with tangency of order $(1-t)^{1+\alpha^\prime}$ where
$\alpha'\leq\alpha$. Consequently, 
the estimates in parts (i), (ii), (iv) and (v) of Lemma~\ref{l.meas} reveal 
 decay of the sets $I_k,J_k$ which is no faster  than for $I_k^\prime,J_k^\prime$.
\end{assumption}

\begin{remark}\label{rem:compare} Below we  use a Young-tower built under the first-return time function to
the set $\Delta_0:=\cup_{l=1}^\infty I_l\cup I^\prime_l$ to prove upper bounds on the
distribution of first $(\sigma,\delta)$-hyperbolic times $h_{\sigma,\delta}$ for certain~$(\sigma,\delta)$. 
Indeed, for $f\in\mathcal{C}_{\alpha}$ ($\alpha<1$), and any $\alpha''\in(\alpha,1)$,
$m\{h_{\sigma,\delta}>n\}=O(n^{-1/\alpha''})$.
The reason for this distribution is that points in $J_l$
(whose Lebesgue measure $\approx l^{-1-1/\alpha}$) require
approximately~$l$ iterates to achieve enough expansion to be a hyperbolic time. Thus, 
$m\{h_{\sigma,\delta}>n\} = \sum_{l>n}m\{h_{\sigma,\delta}=l\}\lesssim l^{\gamma-1/\alpha}$
for any $\gamma>0$.
Similar bounds are used in~\cite{ALP} to prove decay of correlation results by building a Young~tower~$\Delta^\prime$ whose
tail set decays in the same way as the distribution of $h_{\sigma,\delta}$; the resulting\footnote{And when
$0<\alpha<\frac{1}{2}$, a central limit theorem holds~\cite[Theorem 2]{ALP}.}
 decay of correlations are~$O(n^{1-1/{\alpha''}})$---close to the typical rate for maps
with indifferent fixed points with tangencies of $O(x^{1+\alpha})$. Interestingly,
direct calculations in \cite{BM2010b}  where a first-return tower is built over $\Delta_0$ give decay of correlations for 
H\"older observables with rate $O(n^{-1/\alpha})$ for maps in $\mathcal{C}_\alpha$. 
The same computations give a CLT for the entire range $0<\alpha < 1$.   These improved asymptotics are due to the fact that orbits of $f$ experience very rapid expansion when they pass 
near $a$, giving partial compensation for return from the neighbourhoods of $\{0,1\}$,
something that is not accounted for by the hyperbolic times analysis.
\end{remark}

\section{Hyperbolic times and sets of exceptional points}\label{sec.2}

Non-uniformity of expansion is expressed relative to a
certain (finite) set $\mathcal{S} \subseteq [0,1]$ of exceptional
points. The notation here is precisely as in~\cite{ABV,AA,A}.
Each $f$ is locally $C^2$ on $[0,1]\setminus\mathcal{S}$,
and must satisfy a non-degeneracy of the following type:
there exist constants $B>1, \beta >0$ such that
such that for \emph{every} $x \in [0,1] \setminus \mathcal{S}$ we have
\begin{equation}\label{def:exceptional1}
\frac{1}{B} \textnormal{dist}(x, \mathcal{S})^\beta \leq
\left|\frac{df}{dx}\right| \leq
B \, \textnormal{dist}(x, \mathcal{S})^{-\beta}
\end{equation}
and if $y,z \in [0,1]\setminus\mathcal{S}$ and
$|y-z| \leq \frac{\textnormal{dist}(z,\mathcal{S})}{2}$ then
\begin{equation}\label{def:exceptional2}
\left| \log \left.|\textstyle\frac{df}{dx}|\right|_{x=y}-\left.\log|\textstyle\frac{df}{dx}|\right|_{x=z} \right | \leq
\frac{B}{\textnormal{dist}(z, \mathcal{S})^{\beta}} |y-z|
\end{equation}
$\textnormal{dist}(\cdot, \mathcal{S})$ is used to  denote the
usual Euclidean distance to the set $\mathcal{S}$ (since $f$ is
one-dimensional there is no need to impose a separate Lipschitz
condition on $(df/dx)^{-1}$).

\begin{lemma}\label{lem:nondegS1d} 
Under the conditions of Assumption \ref{assum:1}, for $f\in\mathcal{C}_{\alpha}$
set $\beta=1$.  Then there exists a
$B>1$ so that  $\mathcal{S}=\{0,a,1\}$ satisfies
conditions ~(\ref{def:exceptional1})
and~(\ref{def:exceptional2}).  Hence $\mathcal{S}$ is a non-degenerate
set of exceptional points for $f$.

In case $\phi(1)>0$ then only $x=0$ and $x=a$ are required to be exceptional points.
\end{lemma}

\begin{proof} See appendix. \qquad$\Box$
\end{proof}

\subsection*{Hyperbolic times}

Let $\mathcal{S}$ be a non-degenerate exceptional set, $\beta>0$ as in
(\ref{def:exceptional1})~and~(\ref{def:exceptional2}) and fix
$0<b< \textnormal{min}\{ 1/2, 1/(4 \beta)\}$.
Let constants $0<\sigma <1$ and $\delta >0$ be given and define
a truncated distance function
\begin{equation}
\textnormal{dist}_\delta(x, \mathcal{S}):=
\left\{
\begin{array}{ll}
\textnormal{dist}( x, \mathcal{S}) & \textnormal{if}~~ \textnormal{dist}( x, \mathcal{S}) \leq \delta, \\
1& \textnormal{if}~~ \textnormal{dist}( x, \mathcal{S}) > \delta.
\end{array}\right .
\end{equation}
As in~\cite{A,AA,ALP},
$n$ is called a $(\sigma, \delta)$-hyperbolic time for $x \in [0,1]\setminus\mathcal{S}$
if for all $1 \leq l \leq n$
\begin{equation}\label{eqn:defHT}
\prod_{j=n-l}^{n-1}  \left| \left(\textstyle\frac{df}{dx}\circ f^j\right)(x))^{-1}\right| \leq \sigma^l
\qquad\textnormal{and}\qquad
\textnormal{dist}_\delta(f^{n-l}(x), \mathcal{S}) \geq \sigma^{bl}.
\end{equation}

Although orbits escape subexponentially from~$\mathcal{S}$ and the rate of growth of
derivatives along orbits is not uniform, the essential
properties of uniformly expanding maps are captured at hyperbolic times. Since
the invariance (and ergodicity~\cite{B1}) of Lebesgue measure is already at our disposal,
establishing the long-run positive density (in time) of hyperbolic times is relatively straightforward.

\begin{lemma}\label{lem:controlH1} Let $K := -\left(\int_0^{a} \log(\phi_\alpha(f(x)))\, dx
+ \int_{a}^1 \log (1-\phi_\alpha(f(x)))\, dx\right)$. Then
for every $\epsilon >0$ and for every $\sigma\in(e^{-K},1)$,
there exists a $\delta >0$
such that for almost every $x \in [0,1]$
\begin{equation}\begin{array}{cl}
\lim_{N\rightarrow \infty}\frac{1}{N} \sum_{j=0}^{N-1} \log(\textstyle\frac{df}{dx}(f^j(x))^{-1}
&< \log \sigma <0 \\
\\
\lim_{N\rightarrow \infty}\frac{1}{N} \sum_{j=0}^{N-1} -\log \textnormal{dist}_\delta(f^j(x),
\mathcal{S}) & \leq \epsilon.
\end{array}
\end{equation}
\end{lemma}
\begin{proof}  Apply (Birkhoff's) Ergodic Theorem (for the ergodic system $(f, m)$)
to the function $-\log \frac{df}{dx}$ to obtain the first
estimate.  For the second estimate, choose $\delta >0$ such that
$\int_0^1 -\log \textnormal{dist}_\delta(x, \mathcal{S})\,dx < \epsilon$
and apply the Ergodic Theorem again.
\hfill $\Box$
\end{proof}

This lemma says that for a given choice of $b$, there is a good 
choice of $(\sigma,\delta)$ for which there are (many) hyperbolic times
for almost every
$x \in [0,1]\setminus\mathcal{S}$.    
 The  two conditions established in Lemma~\ref{lem:controlH1}
imply that $f$ is a \emph{non-uniformly expanding map} in
the sense of  Alves, Bonatti and Viana \cite{ABV}.  It now follows that:

\newpage

{\bf Positive density of hyperbolic times~{\cite[Lemma~5.4]{ABV}}} {\em
For every $b>0$, for every $\sigma\in(e^{-K},1)$ there exist $\theta >0$
and $\delta >0$ (depending only on $f, \sigma$ and $b$)
so that for almost every $x \in [0,1]$, for all sufficiently large $N$ there
exist $(\sigma, \delta)$-hyperbolic times $1 \leq n_1< \dots < n_l \leq N$ for
x, with $l \geq \theta N$.}

Now that we have established existence of hyperbolic times, we  
define $h_{\sigma, \delta}(x)$ to be the first
$(\sigma, \delta)-$hyperbolic time for $x$. (If there are no $(\sigma,\delta)-$hyperbolic times
for $x$, set $h_{\sigma,\delta}(x)=\infty$.)

The Young tower partition gives simple lower bounds on
$h_{\sigma, \delta}(x)$.  Although crude, these lower bounds are still sufficient for our first main result, Theorem \ref{prop:HtimeLower} below.

\begin{lemma}\label{lem:HtimeLower}
Let $f\in\mathcal{C}_{\alpha}$ and fix $\sigma<1$.
Let $b$ and $\delta>0$ be as above.
Let $k_1$ be minimal such that $\sigma\,\max_{[0,x_{k_1}]\cup [y_{k_1},1]}\frac{df}{dx}< 1$. Then
for $x\in J_k\cup J^\prime_k$ ($k\geq k_1$), $h_{\sigma, \delta}(x)> k-k_1$.
\end{lemma}

\begin{proof}
Since $\frac{df}{dx}$ is increasing on $(0,a)$ and decreasing on $(a,1)$ and
$\lim_{k\to\infty}\frac{df}{dx}|_{x_k}=1=\lim_{k\to\infty}\frac{df}{dx}|_{y_k}$, $k_1$ is well-defined.
Now let $k\geq k_1$, $x\in J_k\cup J^\prime_k$ and fix $n$ with  $1 \leq n\leq k-k_1$. 
For each $j$ with $0 \leq j <n$ we have $f^j(x)\in J_{k-j}\cup J^\prime_{k-j}\subset [0,x_{k_1})\cup(y_{k_1},1]$, so that
$\prod_{j=0}^{n-1} |(\textstyle\frac{df}{dx}\circ f^j(x))^{-1}|> \sigma^{n}$. 
Comparing with~(\ref{eqn:defHT}), $n$ cannot be a $(\sigma,\delta)$-hyperbolic time for $x$.
Since $1 \leq n \leq k-k_1$ was arbitrary, $h_{\sigma,\delta}(x)>k-k_1$.  \qquad$\Box$
\end{proof}

\begin{theorem}[Lower bounds on hyperbolic times]\label{prop:HtimeLower}
Let $f\in\mathcal{C}_{\alpha}$, with $\sigma<1$.
There is a constant $c$ (depending on $\sigma$) such that
$$m\{h_{\sigma,\delta}\geq n\} \geq c\,n^{-1/\alpha}$$
and $h_{\sigma,\delta}(x)$ fails
to be integrable with respect to Lebesgue measure when $\alpha \geq 1$.
\end{theorem}

\begin{proof}
Choose $k_1$ as in Lemma~\ref{lem:HtimeLower}, so $h_{\sigma,\delta}|_{J_{k_1+n}\cup J^\prime_{k_1+n}}> n$.
Hence
$$m\{h_{\sigma,\delta}>n\} \geq \sum_{k=k_1+n}^\infty m(J_k\cup J^\prime_k)\sim n^{-1/\alpha}$$
by Lemma~\ref{l.meas}~(ii). If $\alpha\geq 1$ then
$$\int_0^1h_{\sigma,\delta}\,dm = \sum_{n=1}^\infty m\{h_{\sigma,\delta}\geq n\} = \infty.$$
\hfill $\Box$
\end{proof}

\newpage

\section{Young towers, large deviations and integrability of the first hyperbolic time}\label{sec.3}

Suitable choices of $(\sigma,\delta)$ make $h_{\sigma,\delta}$  integrable when $\alpha<1$. 
To prove this we distinguish
$$\Delta_0=[x_0,y_0] = \cup_{j=1}^\infty I_j\cup I^\prime_j\qquad\pmod{m}$$
as a ``good'' set, where expansion is very rapid and hyperbolic times are easy to control.
The derivative growth condition in~(\ref{eqn:defHT}) is satisfied for $n=1$ on $\Delta_0$,
but for points close to $a$, the condition on $\textnormal{dist}_\delta$ fails for
$n=1$. Controlling $h_{\sigma,\delta}$ involves trading expansion with proximity to
$\mathcal{S}=\{0,a,1\}$, and it turns out that getting enough expansion is the
difficult part.  The idea is to control derivative growth upon successive returns to~$\Delta_0$.
Long excursions near $\{0,1\}$ lead to ``expansivity deficits'' relative to $\sigma^{-n}$,
with ``expansion recovery'' by passage through $J_0\cup J^\prime_0$. This is made
quantitatively precise using a large deviations
result of Melbourne and Nicol~\cite{MN}.

\subsection*{Choice of $\sigma$}

Choose $k_0$ such that $\sum_{k\geq k_0} m(J_k\cup J^\prime_k) < m(J_0\cup J^\prime_0)$.  Now choose $\sigma=\sigma(k_0)<1$ such that
\begin{equation}\label{eqn:setsigma}
\sigma^{2}\,\min_{[x_{1},y_{1}]}\frac{df}{dx}\geq 1 \quad\textnormal{and}\quad
\sigma\,\min_{[x_{k_0},y_{k_0}]}\frac{df}{dx} \geq 1.
\end{equation}
Define
$$N(x)=\left\{\begin{array}{rl} -1 &\textnormal{if~} x\in J_0\cup J'_0,\\
1 &\textnormal{if~} x\in [0,x_{k_0})\cup (y_{k_0},1],\\
0 & \textnormal{otherwise.}\end{array}\right.$$
Notice that
\begin{equation}\label{eqn:Nve}
\int_0^1N\,dm = \sum_{k=k_0}^\infty m(J_k\cup J_k^\prime) - m(J_0 \cup J^\prime_0) < 0.
\end{equation}

\begin{lemma}\label{lem:estH1}
Let $f\in\mathcal{C}_{\alpha}$, $k_0\geq 1$ and put $H=H(x):=\min\{ n\geq 0~:~ \sum_{k=0}^n N\circ f^k(x) < 0\}$. If $0<H<\infty$
then $f^H(x)\in J_0\cup J'_0$ and for all $1\leq l \leq H$
$$\prod_{j=H-l}^{H-1} \left|\textstyle\frac{df}{dx}\circ f^{j}(x)\right|^{-1} =
\left|\textstyle\frac{d(f^l)}{dx}\circ f^{H-l}(x)\right|^{-1} \leq \sigma^{l}.$$
\end{lemma}

\begin{proof} First, for any $0\leq n<H$, $\sum_{k=0}^{n} N\circ f^k(x) \geq 0$. Thus
$$N(f^H(x)) = \sum_{k=0}^H N\circ f^k(x) - \sum_{k=0}^{H-1} N\circ f^k(x) < 0$$ and
hence $N(f^H(x))=-1$ (since $-1$ is the only negative value of $N$).
Thus,\newline $f^H(x)\in J_0\cup J'_0$ and $\sum_{k=0}^{H-1}N\circ f^k(x)=0$. Therefore
$$\sum_{j=H-l}^{H-1} N\circ f^j = \sum_{k=0}^{H-1}N\circ f^k - \sum_{k=0}^{H-l-1}N\circ f^k\leq 0$$
for each $l\leq H$. 
To complete the proof, apply~(\ref{eqn:setsigma}) to notice that for $x\in J_0\cup J^\prime_0$, 
we have $\sigma\,\frac{df}{dx}>\sigma^{-1}$; if $x\in [0,1]\setminus [x_{k_0},y_{k_0}]$ then $\frac{df}{dx}\geq 1$
so $\sigma\,\frac{df}{dx}\geq \sigma$; all other $x$ belong to $[x_{k_0},y_{k_0}]$ so that $\sigma\,\frac{df}{dx}\geq 1$.
In particular, for each type of $x$, $\sigma\,\frac{df}{dx} \geq \sigma^{N(x)}$.
Hence
\begin{eqnarray*}
\prod_{j=H-l}^{H-1} \left|{\textstyle\frac{df}{dx}}\circ f^{j}(x)\right|^{-1} 
&=& \sigma^{l}\prod_{j=H-l}^{H-1} \left|\sigma\,{\textstyle\frac{df}{dx}}\circ f^{j}(x)\right|^{-1}\\
&\leq&
\sigma^{l}\prod_{j=H-l}^{H-1} \sigma^{-N\circ f^j(x)} = \sigma^{l}\,\sigma^{-\sum_{j=H-l}^{H-1}N\circ f^j(x)}
\leq \sigma^{l}.\qquad\Box\end{eqnarray*}
\end{proof}

\subsection*{Choice of $\delta$}

\begin{lemma}\label{lem:setdelta}
Let $f\in\mathcal{C}_{\alpha}$, $\sigma\in(0,1)$ satisfy~(\ref{eqn:setsigma}) and let $b,k_0$ be fixed.
Then there is $\delta>0$ such that whenever
$f^n(x)\in J_0\cup J'_0$ and $1\leq l \leq n$,
$$\textnormal{dist}_\delta(f^{n-l}(x),\{0,a,1\}) \geq \sigma^{bl}.$$
\end{lemma}

\begin{proof} Choose $k_b$ such that
$$x\in I_k\cup I'_k \quad (k\geq k_b) \quad \Rightarrow |x-a| \geq \sigma^{b\,k}$$
(note that this is always possible, since there is a constant $c$ such that
$|x-a| \geq c k^{-1-1/\alpha}$
for all $x\in I^{(\prime)}_{k}$). Choose $\delta$ small enough that $[a-\delta,a+\delta]\subset\cup_{k\geq k_b}(I_k\cup I^\prime_k)$
and $(0,\delta]\cup[1-\delta,1)\subset \cup_{k\geq k_b}(J_k\cup J^\prime_k)$.
For $1\leq l\leq n$ there are three cases to consider:
(i) $f^{n-l}(x)\in I^{(\prime)}_k$ for $k\geq k_b$;
(ii) $f^{n-l}(x)\in J^{(\prime)}_k$ for $k\geq k_b-1$;
(iii) otherwise.
In case (i), for each $j<k$, $f^{n-l+j}(x)\in J^{(\prime)}_{k-j}$. Since $f^n(x)\in J^{(\prime)}_0$
it follows that $n>n-l+(k-1)$ so that $l\geq k$. By the choice of $k_b$,
$$\textnormal{dist}_\delta(f^{n-l}(x),\{0,a,1\})=\max\{1,|f^{n-l}(x)-a|\} \geq \sigma^{b\,k}\geq \sigma^{b\,l}.$$
In case (ii), let $y\in I^{(\prime)}_{k+1}$ be such that $f(y)=f^{n-l}(x)$. Then, since $|\frac{df}{dx}|\geq \sigma^{-1}\geq \sigma^{-b}$ on $\Delta_0$,
$$\textnormal{dist}_\delta(f^{n-l}(x),\{0,a,1\})
=\textnormal{dist}_\delta(f(y),\{0,1\}) \geq  \sigma^{-b}\textnormal{dist}_\delta(y,a) \geq
\sigma^{-b}\sigma^{b(k+1)}\geq \sigma^{bl}.$$
In the final case, $\textnormal{dist}(f^{n-l}(x),\{0,a,1\}) =1>\sigma^{b\,l}$.\qquad$\Box$
\end{proof}

\begin{theorem}\label{thm:isHT}
Let $f\in\mathcal{C}_{\alpha}, b,k_0$ be fixed, let $\sigma$ satisfy~(\ref{eqn:setsigma}) and choose
$\delta$ as in Lemma~\ref{lem:setdelta}. Let $H$ be as defined in Lemma~\ref{lem:estH1}.
Then $\max\{1,H(x)\}$ is a $(\sigma,\delta)$--hyperbolic time for $x$
and if $h_{\sigma,\delta}$ is the first $(\sigma,\delta)$--hyperbolic time then
$$\int_0^1 h_{\sigma,\delta}(x)\,dx \leq  \sum_{n=0}^\infty m\{x~:~H(x)\geq n\}.$$
\end{theorem}

\begin{proof}
If $H=0$ then $x\in J_0\cup J_0^\prime$ and $1$ is a hyperbolic time. Otherwise, comparing ~(\ref{eqn:defHT}) with Lemmas~\ref{lem:estH1} and~\ref{lem:setdelta} shows that $H(x)$ is a hyperbolic time. For the integral,
$$\int_0^1 h_{\sigma,\delta}(x)\,dx \leq \int_0^1 (1+H(x))\,dx = \sum_{n=0}^\infty m\{x~:~H(x)\geq n\}.$$
\hfill $\Box$
\end{proof}

\subsection*{Large deviations on a Young tower}

Let $\Delta_0=[x_0,y_0]$, and partition (modulo sets of measure $0$)
according to $\Delta_{0,n}=I_n\cup I^\prime_n$ for $n>0$. For $x\in\Delta_{0,n}$,
$$R(x):=\min\{k>0~:~f^k(x)\in \Delta_0\} = n+1$$
so we put
$$\Delta:=\cup_{\ell\leq n=0}^\infty (\Delta_{0,n}\times\{\ell\}) \subset \Delta_0\times\mathbb{Z}_+$$
and equip $\Delta$ with the measure $m_\Delta$ obtained by direct upwards translation of $m|_{\Delta_0}$.
The tower map $F:\Delta\circlearrowleft$ is defined in the usual manner: $F(x,\ell)=(x,\ell+1)$ for $\ell<R(x)-1=n$
and $F(x,R(x)-1)=(f^{R(x)}(x),0)$. It is easy to check that
\begin{equation}
\label{eqn.towertailest}
m\{R\geq n\} =\sum_{k> n}m(I_k\cup I^\prime_k)\sim n^{-1-1/\alpha}
\end{equation}
so that $\sum_{k\geq n} m\{R\geq k\} \sim n^{-1/\alpha}$. Standard arguments\footnote{See Lemma 3 in \cite{BM2010b} for example.} show that
branches of the map $f^R$ have uniformly bounded distortion on $\Delta_0$, and since $f^R$ is
uniformly expanding on $\Delta_0$, the tower map $F$ satisfies the usual regularity conditions~\cite{Y2,MN,BM2010b} for
maps on a Young tower\footnote{If $JF$ denotes the Jacobian $\frac{dm_\Delta\circ F}{dm_\Delta}$ then
$\log|JF|\in C_\beta(\Delta)$, where $C_\beta(\Delta)$ is the class of $\beta$--H\"older functions, defined
with respect to the usual~\cite{Y2} separation time~$s$.}. Since Lebesgue measure~$m$ is invariant for $f$,
$m|_{\Delta_0}\circ (f^R)^{-1}=m|_{\Delta_0}$ and hence $m_\Delta$ is actually $F$-invariant on~$\Delta$.
As is usual, $(f,[0,1],m)$ arises as a factor of
$(F,\Delta,m_\Delta)$ via the semi-conjugacy $\Phi(x,\ell)=f^\ell(x)$. Then $\Phi\circ F = f\circ\Phi$, and
$m=\Phi_*m_\Delta=m_\Delta\circ\Phi^{-1}$.

Next, lift $H$ to the tower:
put $\hat{N}(x,\ell):=N\circ{\Phi}(x,\ell)=N(f^\ell(x))$ and
$$\hat{H}:=H\circ\Phi
= \min\left\{n\geq 0~:~ \sum_{k=0}^n\hat{N}\circ F^k < 0\right\}.$$
Let
$$\bar{N} = \int_\Delta \hat{N}\,dm_\Delta = \int_\Delta N\circ \Phi\,dm_\Delta = \int_0^1 Nd(\Phi_*m_\Delta) = \int_0^1\,N\,dm.$$
Note that $\bar{N}<0$ by~(\ref{eqn:Nve}).
Put $\psi:=\hat{N}-\bar{N}$. Then $\psi$ has zero mean, and belongs to every H\"older class $C_\beta(\Delta)$ since it is piecewise constant with respect to the tower partition
(see~\cite{Y2,MN,BM2010b}).

\begin{lemma}[Large Deviations Estimate]\label{lem:LDest}
For every $\epsilon\in(0,-\overline{N})$ and $\gamma>0$ there
is a constant $c$ (not independent of $\gamma$) such that
$$m_\Delta\left\{y\in\Delta~:~ |\textstyle{\frac{1}{n}\sum_{k=0}^{n-1}}\psi\circ F^k(y)|\geq \epsilon\right\}
\leq c\,n^{-1/\alpha + \gamma}.$$
\end{lemma}

\begin{proof}
Using a result of Melbourne and Nicol, 
apply~\cite[Theorem~3.1]{MN} (setting $\delta:=\gamma$ and $\beta:=1/\alpha$) with the observation in~(\ref{eqn.towertailest})
that $m\{y\in\Delta_0~:~R(y)>n\} \sim n^{-1-1/\alpha}$.\qquad$\Box$
\end{proof}

This estimate is enough to prove our main theorem:

\begin{theorem}\label{th:HintCLT}
Let $f\in\mathcal{C}_{\alpha}$ where $\alpha<1$ and let $b<1$ be fixed.
There is a choice of $(\sigma,\delta)$ where $0<\sigma<1$ and $\delta>0$
such that
$$m\{h_{\sigma,\delta}> n\} \leq m\{H\geq n\} = O(n^{-1/\alpha''})$$
for any $\alpha''\in(\alpha,1)$. In particular, the first $(\sigma,\delta)$-hyperbolic time for
$f$ is integrable.
\end{theorem}

\begin{proof}
Let $k_0$ be large enough that $\bar{N}$ is negative and choose $\sigma$ to
satisfy~(\ref{eqn:setsigma}) and $\delta$ as in Lemma~\ref{lem:setdelta}.
The first inequality in the theorem follows because
$h_{\sigma,\delta}\leq H+1$ (compare with Theorem~\ref{thm:isHT}). Next, note that 
\begin{equation}\label{eqn:estH1tower}
m\{H\geq n\} = \Phi_*m_\Delta\{H\geq n\} = m\{H\circ\Phi \geq n\} = m_\Delta\{\hat{H}\geq n\}.
\end{equation}
For $y\in\{\hat{H}\geq n\}$, $\sum_{k=0}^{n-1}\hat{N}\circ F^k(y) \geq 0$ so that
$$ \frac{1}{n}\sum_{k=0}^{n-1}\psi\circ F^k(y) =
\frac{1}{n}\sum_{k=0}^{n-1}(\hat{N}\circ F^k(y) - \bar{N}) \geq -\bar{N}>\epsilon$$
for any $\epsilon\in(0,-\bar{N})$. For each such $\epsilon$,
\begin{equation}\label{eqn:HtailLDset}
\{\hat{H}\geq n\}\subseteq \left\{y\in\Delta~:~ |\textstyle{\frac{1}{n}\sum_{k=0}^{n-1}}\psi\circ F^k(y)|\geq \epsilon\right\}.
\end{equation}
Choosing $\epsilon:=\frac{-\bar{N}}{2}$ and 
$\gamma:=1/\alpha-1/\alpha''$ (where $\alpha''>\alpha$ is  given), by Lemma~\ref{lem:LDest}
there is a constant $c<\infty$ (depending on $\gamma$) such that the set on the RHS of~(\ref{eqn:HtailLDset}) has
$m_\Delta$--measure bounded by $c\,n^{-\frac{1}{\alpha}+\gamma}$.
The second inequality in the theorem now follows from~(\ref{eqn:estH1tower}) and~(\ref{eqn:HtailLDset}).
The integrability of $h_{\sigma,\delta}$ now follows from Theorem~\ref{thm:isHT}.
\hfill  $\Box$
\end{proof}

\section{Conclusions}\label{sec.concl} 

In this paper we have undertaken a detailed study of the asymptotics of \emph{hyperbolic times} for
a parameterized family $f_\alpha \in \mathcal{C}_\alpha, \, 0< \alpha < \infty$ of non-uniformly expanding, 
Lebesgue-measure-preserving maps of the interval. These 
one-dimensional maps arise naturally as the expanding factors of a class of two-dimensional \emph{generalized baker's transformations} on the unit square. 

A central result in the literature, Alves, Bonatti and Viana \cite{ABV}, proves that if a non-uniformly hyperbolic map has the property that  almost every $x$ has a (uniform) positive frequency of hyperbolic times, then it admits an absolutely continuous invariant measure.   According to Lemma \ref{lem:controlH1}, each of our maps $f_\alpha$ has this property; of course the resulting invariant measure is already known to be Lebesgue. 

Therefore our main result concerns the statistics of \emph{first hyperbolic times} $h_\alpha$ for
our maps $f_\alpha$ and in particular how the quantities $m\{x ~:~ h_\alpha(x) > n\}$ depend on 
$n$ and $\alpha$. We show that this is entirely determined by the strength of the (most indifferent) fixed point for the map $f_\alpha$. In particular, 
\begin{itemize}
\item $h_\alpha$ is integrable if and only if $0 < \alpha < 1$, corresponding to relatively fast polynomial escape from the indifferent fixed points at $x=0,1$ within the range of our family of maps $\mathcal{C}_\alpha,  \,0 < \alpha < \infty$  (Theorems~\ref{prop:HtimeLower} and~\ref{th:HintCLT}).
\item For $0< \alpha < 1$ and $\alpha^{\prime \prime} \in (\alpha, 1)$ we establish $m\{ h_\alpha >n\} = O(n^{-\frac{1}{\alpha^{\prime \prime}}})$ which implies, by using results from 
\cite{ALP2,ALP}, correlation decay for H\"older observables\footnote{See, for example Theorem 3 of \cite{ALP2}.} at rate $O(n^{-\frac{1}{\alpha^{\prime \prime}} + 1})$. 
\item  Finally,  a CLT holds in case $0 < \alpha< 1/2$, again by applying 
the results of \cite{ALP2,ALP} to the case $\alpha^{\prime \prime} \in (\alpha, 1/2)$. 
\end{itemize}
The conclusions which are obtained in \cite{ALP2, ALP} depend on the construction of a suitable Markov or Young tower via first hyperbolic times and their statistics and the analysis of the return times on the resulting tower. 

On the other hand, \cite{BM2010b}  details a specific construction of a Young tower for the maps $f_\alpha$ and proves an improved 
correlation decay rate  of $O(n^{-\frac{1}{\alpha}})$ for all $0<\alpha < \infty$. It is shown that this correlation decay rate is sharp for H\"older data. 
The estimates in \cite{BM2010b} also imply a CLT for $0 < \alpha < 1$.  See Remark \ref{rem:compare} for further discussion. 

We interpret this as follows.  The analysis via hyperbolic times provides a relatively general approach to analysis of nonuniformly hyperbolic systems that leads to a particular Young tower construction related to the hyperbolic times. It is not so surprising that this approach does not always yield optimal results such as sharp estimates on decay of correlation rates or the CLT.   
In the case of our maps $f_\alpha$, for example, a dedicated tower construction in \cite{BM2010b}
can produce optimal results in the form of sharp estimates on correlation decay by bypassing the intermediate construction of hyperbolic times.

\section{Appendix}\label{sec.4}

{\em Proof of Lemma~\ref{lem.setup}:\/} (i) The first equation is immediate from Equation (\ref{eqn.map1}) For the second, again use Equation (\ref{eqn.map1}) and write
$$ x_r=1-\int_x^1(1-\phi(t))\,dt=x +\int_x^1\phi(t)\,dt=x + a -\int_0^x\phi(t)dt =  a + \int_0^x 1 -\phi(t) \, dt $$
and so
$$x_r -a  = \int_0^x 1-\phi(t) \, dt = x - x_l.$$
(ii) Differentiating the expressions in (i) via Lebesgue's theorem gives:
$$ \frac{dx_l}{dx}= \phi \textnormal{ and } \frac{dx_r}{dx} = 1- \phi.$$
(iii) Apply (i) and (ii):
$$m(T^{-1}[0,x]) = \int_0^{x_l} \1 dt + \int_a^{x_r} \1 dt = \int_0^x \frac{dx_l}{dx} dx + \int_0^x  \frac{dx_r}{dx}\, dx = \int_0^x \1 dx= m[0,x].
 \qquad\Box$$

{\em Proof of Lemma~\ref{l.meas}~(v):\/} For this part fix $k$ and $x \in I_k$. Since $x^\prime_n\to a$ as $n\to\infty$,
$$\textnormal{dist}(x,a) = x-a = x-x_{{k+1}}+\sum_{j={k+1}}^\infty(x_{j}-x_{j+1})=a-x_{k+1}+\sum_{j=k+1}^\infty m(I_j).$$
In fact this argument shows that
$\sum_{j=k+1}^\infty m(I_j) \leq \textnormal{dist}(x, a) \leq
\sum_{j=k}^\infty m(I_j)$
and both sides are $\sim\sum_{j=k}^\infty j^{-2-1/\alpha} \sim k^{-1-1/\alpha}$.\qquad\qquad$\Box$

{\em Proof of Lemma~\ref{lem:nondegS1d}:\/} First, note that 
$df/dx=|df/dx|\geq 1 $, 
and $\textnormal{dist}(z,\mathcal{S})<\max\{a,1-a\}/2$ for every~$z$,
so for the lower bound in~(\ref{def:exceptional1}) holds whenever $B>B_\beta:=((\max\{a,1-a\}/2)^{-\beta}$.  We will establish the 
right hand side inequality after proving (\ref{def:exceptional2})
 For that, it suffices
to work on $(0,a)$, since the other interval is similar. By~(\ref{eqn.mapd}), $\frac{df}{dx}=1/\phi\circ f$,
and hence
$$\frac{d}{dx}\log\frac{df}{dx} =-\left(\frac{d\phi/dx}{\phi^2}\right)\circ f.$$
If $t\in J_k$ ($k>0$) then $f(t)\in J_{k-1}$ so $f(t)\sim k^{-1/\alpha}$ and (since $\phi$ is decreasing), $1\geq \phi(f(t))\geq \phi(x_0)$.
Moreover, $\frac{d\phi}{dx}|_{x=f(t)}=-\alpha\,c_0(f(t))^{\alpha-1}+o((f(t))^{\alpha-1})\sim - k^{1/\alpha-1}$.  Hence
$\left|\frac{d}{dx}\log\frac{df}{dx}\right|\sim k^{1/\alpha-1}$. If $t\in J_0$ then $f(t)\in[x_0,y_0]$ and $\phi(f(t))\sim 1$ (since
$\phi$ is decreasing and $C^1$). If $t\in I^\prime_k$ then $f(t)\in J^\prime_{k-1}$ so $1-f(t)\sim k^{-1/\alpha'}$, $\phi(f(t))\sim k^{-1}$
and $\frac{d\phi}{dx}|_{x=f(t)}\sim k^{1/\alpha' -1}$.
Hence
$$\left|\frac{d}{dx}\log\frac{df}{dx}\right|(t) \sim\left\{ \begin{array}{ll}
1& t\in J_0,\\
k^{1/\alpha-1} & t\in J_k (k>0),\\
k^{1/\alpha'+1} & t\in I^\prime_k,\end{array}\right.
\sim
\left\{ \begin{array}{ll}
1& t\in J_0,\\
\textnormal{dist}(t,\mathcal{S})^{-(1-\alpha)} & t\in J_k (k>0),\\
\textnormal{dist}(t,\mathcal{S})^{-1} & t\in I^\prime_k.\end{array}\right.
$$
(since $\textnormal{dist}(t,\mathcal{S})|_{J_k}\sim\textnormal{dist}(t,0)|_{J_k}\sim k^{-1/\alpha}$
and $\textnormal{dist}(t,\mathcal{S})|_{I^\prime_k}\sim\textnormal{dist}(t,a)|_{I^\prime_k}\sim k^{-1/\alpha'-1}$
by Lemma~\ref{l.meas}). For any $\beta\geq 1$
there is a constant $B_0$ such that
\begin{equation}\label{e:regest}
\left|\frac{d}{dx}\log\frac{df}{dx}\right| (t) \leq B_0 \,\textnormal{dist}(t,\mathcal{S})^{-\beta}.
\end{equation}
To complete the proof of~(\ref{def:exceptional2}) let $|y-z|<\frac{\textnormal{dist}(z,\mathcal{S})}{2}$. Then
the mean value theorem gives a $t$ between $y,z$ such that
$$\log \left|\frac{df/dx|_{x=y}}{df/dx|_{x=z}} \right | = \left|\frac{d}{dx}\log\frac{df}{dx}\right|(t)\times|y-z|
\leq \frac{B_0}{\textnormal{dist}(t, \mathcal{S})^{\beta}}\,|y-z|$$
(using~\ref{e:regest})).
But since $|t-z|\leq|z-y|\leq \frac{\textnormal{dist}(z,\mathcal{S})}{2}$,
$\textnormal{dist}(t,\mathcal{S})\geq \textnormal{dist}(z,\mathcal{S})/2$
so choosing $B=B_0\,2^\beta$ completes the regularity estimate. A similar argument estimating
$\frac{df}{dx}= 1/\phi\circ f$  on the interval $[x_0, a)$ gives
$$|df/dx| = O\left(\textnormal{dist}(x,\mathcal{S})^{-\alpha'/(1+\alpha')}\right),$$
giving the upper bound in~(\ref{def:exceptional1}) for any $\beta\geq1$.\qquad\qquad$\Box$

\end{document}